\documentclass[12pt,reqno]{article}

\usepackage{amsmath,amsthm,amsfonts,amssymb,latexsym,mathrsfs,color}
\usepackage[colorlinks=true,
linkcolor=webblue, filecolor=webbrown, citecolor=webred ]{hyperref}
\definecolor{webblue}{rgb}{0,.5,0}
\definecolor{webred}{rgb}{0,.5,0}
\definecolor{webbrown}{rgb}{.6,0,0}

\newtheorem{thm}{Theorem}[section]
\newtheorem{lem}[thm]{Lemma}
\newtheorem{cor}[thm]{Corollary}
\newtheorem{prop}[thm]{Proposition}

\newtheorem{cl}{Claim}

\theoremstyle{definition}

\newtheorem{ex}[thm]{Example}
\newcommand{\D}{d}
\renewcommand{\d}{\delta}
\newtheorem{rem}[thm]{Remark}

\numberwithin{equation}{section}

\def\la{\lambda}

\title{Higher order log-monotonicity of combinatorial sequences
\thanks{Supported partially by the National Natural Science Foundation of China (Nos.11071030,
11201191), the NSF of Jiangsu Higher Education Institutions
(No.12KJB110005) and the PAPD of Jiangsu Higher Education
Institutions.
\newline\hspace*{5mm}
   {\it Email addresses:} bxzhu@jsnu.edu.cn (B.-X. Zhu)}}
\author{Bao-Xuan Zhu}
\date{\footnotesize School of Mathematics and Statistics,
         Jiangsu Normal University,
         Xuzhou 221116, PR China}
\begin{document}

\maketitle

\begin{abstract}
A sequence $\{z_n\}_{n\geq0}$ is called ratio log-convex in the
sense that the ratio sequence $\{\frac{z_{n+1}}{z_n}\}_{n\geq0}$ is
log-convex. Based on a three-term recurrence for sequences, we
develop techniques for dealing with the ratio log-convexity of ratio
sequences. As applications, we prove that the ratio sequences of
numbers, including the derangement numbers, the Motzkin numbers, the
Fine numbers, Franel numbers and the Domb numbers are ratio
log-convex, respectively. Finally, we not only prove that the
sequence of derangement numbers is asymptotically infinitely
log-monotonic, but also show some infinite log-monotonicity of some
numbers related to the Gamma function, in particular, implying two
results of Chen {\it et al.} on the infinite log-monotonicity of the
Catalan numbers and the central binomial coefficients.
\bigskip\\
{\sl MSC:}\quad 05A20; 11B68
\\
{\sl Keywords:}\quad Ratio sequence; Log-convexity; Ratio
log-convexity; Log-concavity
\end{abstract}
\section{Introduction}
Let $\{z_n\}_{n\geq0}$ be a sequence of positive numbers. It is
called {\it log-concave} (resp. {\it log-convex}) if
$z_{n-1}z_{n+1}\le z_n^2$ (resp. $z_{n-1}z_{n+1}\ge z_n^2$) for all
$n\ge 1$. Clearly, the sequence $\{z_n\}_{n\ge 0}$ is log-concave
(resp. log-convex) if and only if the sequence
$\{z_{n+1}/z_n\}_{n\ge 0}$ is decreasing (resp. increasing).
 Motivated by a conjecture of F. Firoozbakht on
monotonicity of the sequence $\{\sqrt[n]{p_n}\}_{n\ge 1}$ for the
$n$th prime $p_n$, Sun~\cite{Sun-conj} posed a series of conjectures
about monotonicity of sequences of the forms
$\{\sqrt[n]{z_n}\}_{n\ge 1}$ and
$\{\sqrt[n+1]{z_{n+1}}/\sqrt[n]{z_n}\}$, where $\{z_n\}_{n\ge 0}$ is
a familiar number-theoretic or combinatorial sequence. In fact,
there are certain natural links between these two sequences
$\{z_{n+1}/z_n\}$ and $\{\sqrt[n]{z_n}\}$. For instance, it is well
known that if the sequence $\{z_{n+1}/z_n\}$ is convergent, then so
is the sequence $\{\sqrt[n]{z_n}\}$. Recently, Wang and
Zhu~\cite{WZ13} also found a similar result for monotonicity: if the
sequence $\{z_{n+1}/z_n\}$ is increasing (decreasing), then so is
the sequence $\{\sqrt[n]{z_n}\}$ when $z_0\le 1$ ($z_0\ge 1$). In
addition, Chen et al.~\cite{CGW1} further proved if the sequence
$\{z_{n+1}/z_n\}$ is log-concave (log-convex), then so is the
sequence $\{\sqrt[n]{z_n}\}$ under a certain initial condition.
Thus, most conjectures of Sun on the monotonicity of
$\{\sqrt[n]{z_n}\}_{n\ge 1}$ and
$\{\sqrt[n+1]{z_{n+1}}/\sqrt[n]{z_n}\}$ could be solved in certain
unified approach. Some other special results for monotonicity of
this kind can be found in Chen {\it et al.}~\cite{CGW}, Hou {\it et
al.}~\cite{HSW12}, Luca and St\u{a}nic\u{a}~\cite{LS12}, and
Zhu~\cite{Z13}. The log-behaviors of sequences arise often in
combinatorics, algebra, geometry, analysis, probability and
statistics and have been extensively investigated, see
\cite{Sta89,Bre94,LW07, WY07,Zhu12} for instance.

Define an operator $R$ on a sequence $\{z_n\}_{n\geq 0}$ by
$$R\{z_n\}_{n\geq 0}=\{x_n\}_{n\geq 0},$$ where $x_n = z_{n+1}/{z_n}$. The sequence
$\{z_n\}_{n\geq 0}$ is called {\it log-monotonic of order $k$} if
for $r$ odd and not greater than $k-1$, the sequence
$R^r\{z_n\}_{n\geq 0}$ is log-concave and for $r$ even and not
greater than $k-1$, the sequence $R^r\{z_n\}_{n\geq 0}$ is
log-convex. The sequence $\{z_n\}_{n\geq 0}$ is called {\it
infinitely log-monotonic} if it is log-monotonic of order $k$ for
all integers $k\geq0$. Note that a sequence $\{z_n\}_{n\geq0}$ is
log-monotonic sequence of order two if and only if it is log-convex
and the ratio sequence $\{\frac{z_{n+1}}{z_n}\}_{n\geq0}$ is
log-concave. This stimulates some new concepts as follows: A
sequence $\{z_n\}_{n\geq0}$ is said to be {\it ratio log-concave}
(resp. {\it ratio log-convex} ) if
$\{\frac{z_{n+1}}{z_n}\}_{n\geq0}$ is log-concave (resp.
log-convex).

After proved ratio log-concavity of some combinatorial sequences,
Chen et al.~\cite{CGW1} proposed some conjectures that the sequences
of the Motzkin numbers, the Fine numbers, the central Delannoy
numbers and the Domb numbers, respectively, are almost infinitely
log-monotonic in the sense that for each integer $k\geq 0$, these
sequences are log-monotonic of order $k$ except for some entries at
the beginning. Motivated by these conjectures, we further consider
the log-monotonic of order three. As applications, we prove that the
ratio sequences of numbers, including the derangement numbers, the
Motzkin numbers, the Fine numbers, Franel numbers and the Domb
numbers are ratio log-convex, respectively, see Section $2$. On the
other hand, Chen et al.~\cite{CGW1} also proved that the Catalan
numbers and the central binomial coefficients are infinitely
log-monotonic. These results stimulate some further observations for
infinite log-monotonicity of derangement numbers and some binomial
coefficients, see Section $3$.

\section{Ratio Log-convexity of Ratio
Sequences}
In this section, we will consider log-monotonicity of order three.
It is clear that a sequences $\{z_n\}_{n\geq 0}$ is log-monotonic
sequence of order three if and only if it is log-convex and ratio
log-concave, and the ratio sequence
$\{\frac{z_{n+1}}{z_n}\}_{n\geq0}$ is ratio log-convex.
\begin{thm}\label{thm+}
Let $\{z_n\}_{n\ge 0}$ be a sequence of positive numbers satisfying
the recurrence
\begin{equation}\label{rr-cc+}
z_{n+1}=a_nz_{n}+b_nz_{n-1}
\end{equation}
for $n\ge 1$. Assume for $n\geq 1$ that $b_{n+1}\geq b_n>0$,
$a_{n+1}\geq a_n>0$ and $$21a_n^2+11 a_{n+1}a_n-4b_{n-1}\geq0.$$ Let
the function
\begin{equation}\label{function}
f(x)=\left[(a_{n+1}a_n+b_{n+1})x+a_{n+1}b_n\right](x-a_{n-1})x^6-b_{n-1}(a_nx+b_n)^4.
\end{equation}
If there exists a positive integer $N$ and a sequence $\{\la_n\}$
such that $\frac{z_{n}}{z_{n-1}}\geq \la_n\geq a_n$, $f''(\la_n)>0$,
$f'(\la_n)>0$, and $f(\la_n)>0$ for all $n\geq N+1$, then the ratio
sequence $\{\frac{z_{n+1}}{z_n}\}_{n\ge N}$ is ratio log-convex.
\end{thm}
\begin{proof}
In order to prove that the ratio sequence
$\{\frac{z_{n+1}}{z_n}\}_{n\ge N}$ is ratio log-convex, it suffices
to show that
\begin{eqnarray*}
z_{n+2}z_{n-2}z_n^6-z_{n+1}^4z_{n-1}^4>0
\end{eqnarray*}
for $n>N+1$. By the recurrence~(\ref{rr-cc+}), we get that
\begin{eqnarray*}
&&z_{n+2}z_{n-2}z_n^6-z_{n+1}^4z_{n-1}^4\\
&=&(a_{n+1}z_{n+1}+b_{n+1}z_n)\frac{z_n-a_{n-1}z_{n-1}}{b_{n-1}}
z_n^6-(a_nz_{n}+b_nz_{n-1})^4z_{n-1}^4\\
&=&\frac{(a_{n+1}(a_nz_{n}+b_nz_{n-1})+b_{n+1}z_n)(z_n-a_{n-1}z_{n-1})z_n^6-b_{n-1}(a_nz_{n}+b_nz_{n-1})^4z_{n-1}^4}{b_{n-1}}\\
&=&\frac{z_{n-1}^8}{b_{n-1}}\left[(a_{n+1}(a_n\frac{z_{n}}{z_{n-1}}+b_n)+b_{n+1}\frac{z_{n}}{z_{n-1}})(\frac{z_{n}}{z_{n-1}}
-a_{n-1})(\frac{z_{n}}{z_{n-1}})^6-b_{n-1}(a_n\frac{z_{n}}{z_{n-1}}+b_n)^4\right]\\
&=&\frac{z_{n-1}^8}{b_{n-1}}f(\frac{z_{n}}{z_{n-1}}).
\end{eqnarray*}
Thus, it suffices to prove $f(\frac{z_{n}}{z_{n-1}})>0$ since
$b_{n-1}>0$.
 Noticing that
\begin{equation}
f(x)=\left[(a_{n+1}a_n+b_{n+1})x+a_{n+1}b_n\right](x-a_{n-1})x^6-b_{n-1}(a_nx+b_n)^4,
\end{equation}
we can deduce that
\begin{eqnarray*}
f^{(3)}(x)&=&8\times 42(a_{n+1}a_n+b_{n+1})x^5+5\times 42
\left[a_{n+1}b_n-a_{n-1}(a_{n+1}a_n+b_{n+1})\right]x^4\\
&&-120 a_{n+1}a_{n-1}b_nx^3-24b_{n-1}a_n^3(a_nx+b_n)\\
&=& 42(a_{n+1}a_n+b_{n+1})x^4(8x-5a_{n-1})+30
a_{n+1}b_nx^3(7x-4a_{n-1})-24b_{n-1}a_n^3(a_nx+b_n)\\
&\geq&3\times42(a_{n+1}a_n+b_{n+1})x^5+90
a_{n+1}b_nx^4-24b_{n-1}a_n^3(a_nx+b_n)\\
&\geq&3\times42(a_{n+1}a_n+b_{n+1})a_n^5+66
a_{n+1}b_nx^4-24b_{n-1}a_n^3b_n\\
&>&6a_n^3\left[21(a_{n+1}a_n+b_{n+1})a_n^2+11
a_{n+1}b_na_n-4b_{n-1}b_n\right]\\
&>&6a_n^3b_n\left[21a_n^2+11 a_{n+1}a_n-4b_{n-1}\right]\\
&\geq&0
\end{eqnarray*}
for $x\geq a_{n}$ since $b_{n+1}\geq b_n>0$ and $a_{n+1}\geq a_n>0$
for $n\geq 1$. Thus $f''(x)$ is strictly increasing for $x\geq
a_{n}$, which implies that $f''(x)\geq f''(\la_n)>0$ for $x\geq
\la_n\geq a_{n}$. Hence, it follows that $f'(x)$ is strictly
increasing for $x\geq \la_{n}$ and $f'(x)\geq f'(\la_n)>0$.
Therefore, we conclude that $f(x)$ is strictly increasing for $x\geq
\la_n$. Thus, it follows from $\frac{z_{n+1}}{z_n}\geq \la_n\geq
a_n$ that $f(\frac{z_{n+1}}{z_n})\geq f(\la_n)>0$, as desired. This
completes the proof.
\end{proof}
The lower bound $\la_n$ on $\frac{z_{n}}{z_{n-1}}$ in Theorem
\ref{thm+} can be obtained by the next result.
\begin{lem}\emph{\cite{CGW1}}
Let $\{z_n\}_{n\geq0}$ be the sequence defined by the recurrence
relation (\ref{rr-cc+}). Assume that $b(n) > 0$ for $n\geq1$. If
there exists a positive integer $N$ and a sequence $g_n$ such that
$$g_N<\frac{z_N}{z_{N-1}}<\frac{b_{N}}{g_{N+1}-a_{N}}$$
and the inequality
$$a_{n-1}+\frac{b_{n-1}}{g_{n-1}}<\frac{b_n}{g_{n+1}-b_n}$$
holds for all $n\geq N$, then for $n\geq N$,
$$g_{n}<\frac{z_n}{z_{n-1}}<\frac{b_n}{g_{n+1}-a_n}.$$
\end{lem}

In what follows, we give some applications of Theorem~\ref{thm+}.
\begin{ex}
The derangements number $d_n$ is the number of permutations of $n$
elements with no fixed points. It is well known that the sequence
$\{d_n\}_{n\ge 0}$ satisfies the recurrence
$$d_{n+1}=n(d_{n}+d_{n-1}),$$ with $d_0=1,d_1=0,d_2=1,d_3=2$ and
$d_4=9$, see Comtet~\cite[p.~182]{Com74}. The next result can be
proved by Theorem~\ref{thm+}.
\begin{prop}\label{prop+d}
For derangements number $d_n$, the ratio sequence
$\{\frac{d_{n+1}}{d_n}\}_{n\ge 0}$ is ratio log-convex.
\end{prop}
\begin{proof}
Set $a_n=n$ and $b_{n}=n$. Clearly, they are positive and strictly
increasing, and
\begin{eqnarray*}
21a_n^2+11 a_{n+1}a_n-4b_{n-1}=32n^2-4n>0
\end{eqnarray*}
for $n\geq1$. On the other hand, it is not hard to obtain that the
function
\begin{eqnarray*}
f(x)&=&\left[(a_{n+1}a_n+b_{n+1})x+a_{n+1}b_n\right](x-a_{n-1})x^6-b_{n-1}(a_nx+b_n)^4\\
&=&(n+1)[(n+1)x+n](x-n+1)x^6-(n-1)n^4(x+1)^4.
\end{eqnarray*}
Thus, it follows that
\begin{eqnarray*}
f(a_n)&=&(1 + n)(1152 + 3968 n + 5568n^2 + 4048n^3 + 162n^4 +
        362n^5 + 4n^6 + 2n^7)>0,\\
 f'(a_n)&=&(2 + n)(2272 + 8080n + 11776 n^2 + 9164 n^3 + 4206n^4 +
        1201n^5 + 214n^6 \\
&&+ 22 n^7 + n^8)>0,\\
 f''(a_n)&=&2(2 + n)^2(2056 + 5480 n + 5656 n^2 + 2926n^3 +
        820n^4 + 119 n^5 + 7n^6) >0.
 \end{eqnarray*} So, by Theorem~\ref{thm+}, we obtain that the ratio sequence
$\{\frac{d_{n+1}}{d_n}\}_{n\ge 0}$ is ratio log-convex. This
completes the proof.
\end{proof}
\end{ex}
\begin{ex}
The Motzkin number $M_n$ counts the number of lattice paths, {\it
Motzkin paths}, starting from $(0,0)$ to $(n,0)$, with steps $(1,0),
(1,1)$ and $(1,-1)$, and never falling below the $x$-axis, or
equally, the number of lattice paths from  $(0,0)$ to $(n,n)$ with
steps $(0,2),(2,0)$ and $(1,1)$, never rising above the line $y=x$.
It is known that the Motzkin numbers satisfy the recurrence
\begin{equation}\label{rr-m}
(n+3)M_{n+1}=(2n+3)M_n+3nM_{n-1},
\end{equation}
with $M_0=M_1=1$, see \cite{Sul01} for a bijective proof. It follows
from Theorem~\ref{thm+} that the following result can be verified.
\begin{prop}
The ratio sequence $\{\frac{M_{n+1}}{M_n}\}_{n\ge 0}$ is ratio
log-convex.
\end{prop}
\begin{proof}
By the recurrence (\ref{rr-m}), we have $a_n=\frac{2n+3}{n+3}$ and
$b_{n}=\frac{3n}{n+3}$, which are positive and strictly increasing,
and $21a_n^2+11 a_{n+1}a_n-4b_{n-1} >0$ for $n\geq1$. On the other
hand, it is easy to obtain that the function
\begin{eqnarray*}
f(x)&=&\left[(a_{n+1}a_n+b_{n+1})x+a_{n+1}b_n\right](x-a_{n-1})x^6-b_{n-1}(a_nx+b_n)^4\\
&=&\frac{(7n^2+28n+24)x+3(2n+5)}{(n+3)(n+4)}(x-\frac{2n+1}{n+2})x^6-\frac{(3n-3)[(2n+3)x+3n]^4}{(n+2)(n+3)^4}.
\end{eqnarray*}
Further, from Chen {\it et al.} \cite{CGW1}, we have
$$\frac{M_{n}}{M_{n-1}}\geq \frac{6n^2 +3n - 8/9}{2n(n+2)}>a_n.$$ Thus,
taking $\la_n=\frac{6n^2 +3n - 8/9}{2n(n+2)}$ and using Maple, we
can get that $f''(\la_n)>0$, $f'(\la_n)>0$ and $f(\la_n)>0$ for
$n\geq1$, as desired. Hence, by Theorem~\ref{thm+}, the ratio
sequence $\{\frac{M_{n+1}}{M_n}\}_{n\ge 0}$ is ratio log-convex.
This completes the proof.
\end{proof}
\end{ex}
\begin{ex}
The Fine number $f_n$ is the number of Dyck paths from $(0,0)$ to
$(2n,0)$ with no hills. The Fine numbers satisfy the recurrence
\begin{equation}\label{rr-f}
2(n+2)f_{n+1}=(7n+2)f_{n}+2(2n+1)f_{n-1},
\end{equation}
with $f_0=1$ and $f_1=0$, see \cite{PW99} for a bijective proof.
\begin{prop}
The ratio sequence $\{\frac{f_{n+1}}{f_n}\}_{n\ge 0}$ is ratio
log-convex.
\end{prop}
\begin{proof}
Note that $a_n=\frac{(7n+2)}{2(n+2)}$ and
$b_{n}=\frac{(2n+1)}{n+2}$. Obviously, they are positive and
strictly increasing and $21a_n^2+11 a_{n+1}a_n-4b_{n-1} >0$ for
$n\geq1$. On the other hand, from Liu and Wang \cite{LW07}, we have
$$\frac{f_{n}}{f_{n-1}}\geq \frac{4n+6}{n+3}>a_n.$$ Thus,
taking $\la_n=\frac{4n+6}{n+3}$ and substituting
$a_n=\frac{(7n+2)}{2(n+2)}$ and $b_{n}=\frac{(2n+1)}{n+2}$ into
(\ref{function}), we can easily get that $f''(\la_n)>0$,
$f'(\la_n)>0$ and $f(\la_n)>0$ for $n\geq2$ by using Maple. Hence,
by Theorem~\ref{thm+}, we obtain that the ratio sequence
$\{\frac{f_{n+1}}{f_n}\}_{n\ge 0}$ is ratio log-convex. This
completes the proof.
\end{proof}
\end{ex}
\begin{ex}
The Franel numbers $F_n=\sum_{k=0}^n\binom{n}{k}^3$ and they also
satisfy the recurrence
$$(n+1)^2F_{n+1}=(7n(n+1)+2)F_n+8n^2F_{n-1}$$ for all $n\geq 1$
with the initial values given by $F_0=1,F_1=2,F_2=10$, see
\cite[A000172 ]{Slo}. Similarly, by Theorem~\ref{thm+}, we can show
the following result (we leave the details to the reader), where we
can take $\la_n=\frac{8n^2+8n+1}{(n+1)^2}$.
\begin{prop}
The ratio sequence $\{\frac{F_{n+1}}{F_n}\}_{n\ge 0}$ is ratio
log-convex.
\end{prop}
\end{ex}
For another case that $b(n) < 0$ for all $n\geq1$ in the recurrence
(\ref{rr-cc+}), similar to the proof of Theorem~\ref{thm+}, we can
also prove the following criterion, whose proof is omitted for
brevity.
\begin{thm}\label{thm-}
Let $\{z_n\}_{n\ge 0}$ be a sequence of positive numbers satisfying
the recurrence
\begin{equation}\label{rr-c-}
z_{n+1}=a_nz_{n}+b_nz_{n-1}
\end{equation}
for $n\ge 1$, where $a_n>0$ and $b_n<0$. Assume that
$a_{n+1}a_n+b_{n+1}>0$ for $n\geq 1$. Let the function
\begin{equation}\label{Function}
f(x)=\left[(a_{n+1}a_n+b_{n+1})x+a_{n+1}b_n\right](x-a_{n-1})x^6-b_{n-1}(a_nx+b_n)^4.
\end{equation}
If there exists two sequences $r(n)$ and $s(n)$ and a positive $N$
such that, for all $n\geq N$,
\begin{enumerate}
\item [\rm(i)]
$r(n)\leq\frac{z_{n}}{z_{n-1}}\leq s(n)\leq a_n$,
\item [\rm(ii)]
$8(a_{n+1}a_n+b_{n+1})r(n)+5
(a_{n+1}b_{n}-a_{n-1}a_{n+1}a_n-a_{n-1}b_{n+1})\geq0$ and
$a_nr(n)+b_n\geq0$,
\item [\rm(iii)]
$f''(r(n))>0$, $f'(r(n))>0$ and $f(s(n))<0$,
\end{enumerate} then the ratio sequence
$\{\frac{z_{n+1}}{z_n}\}_{n\ge N}$ is ratio log-convex.
\end{thm}
\begin{rem}
The lower bound $r(n)$ for $\frac{z_{n}}{z_{n-1}}$ in Theorem
\ref{thm-} can be obtained by the log-convexity property of $z_n$.
The upper bound $\la_n$ on $\frac{z_{n}}{z_{n-1}}$ in Theorem
\ref{thm-} can be obtained by the next result.
\end{rem}
\begin{lem}\emph{\cite{CGW1}}
Let $\{z_n\}_{n\geq0}$ be the sequence defined by the recurrence
relation (\ref{rr-c-}). Assume that $b(n)< 0$ for $n\geq1$. If there
exists a positive integer $N$ and a sequence $h_n$ such that
$\frac{z_N}{z_{N-1}}<h_N$ and the inequality
$$h_{n+1}>a_{n}+\frac{b_{n}}{h_{n}}$$
holds for all $n\geq N$, then $\frac{z_n}{z_{n-1}}<h_n$ for $n\geq
N$.
\end{lem}

Now we give an example as follows.
\begin{ex}
The Domb numbers $D_n$ are given by
$$D_n=\sum_{k=0}^n\binom{n}{k}^2\binom{2k}{k}\binom{2(n-k)}{n-k}$$
for all $n\geq0$. The first few terms in the sequence of Domb
numbers $\{D_n\}_{n\ge0}$ are as follows:
$$1,4,28,256,2716,\ldots,$$
see \cite[A002895]{Slo}. They also satisfy the recurrence
$$(n+1)^3D_{n+1}=2(2n+1)(5n(n+1)+2)D_n-64n^3D_{n-1}$$ for all $n\geq 1$.

\begin{prop}
For the Domb numbers $D_n$, the ratio sequence
$\{\frac{D_{n+1}}{D_n}\}_{n\ge 0}$ is ratio log-convex.
\end{prop}
\begin{proof}
Let $a_n=\frac{2(2n+1)(5n(n+1)+2)}{(n+1)^3}$ and
$b_{n}=\frac{-64n^3}{(n+1)^3}$. It is easy to check
$$a_{n+1}a_n+b_{n+1}=\frac{16(716 + 2940 n + 5304 n^2 + 4605 n^3 +
1935n^4 + 351n^5 + 21n^6)}{(1 + n)^3(2 + n)^3}>0$$ for $n\geq 1$.
From Chen {\it et al.} \cite{CGW1}, we have
$$ \frac{D_{n}}{D_{n-1}}\leq\frac{16n^3 - 24n^2 + 12n - 2}{n^3}\leq a_n.$$
On the other hand, by the log-convexity of $D_n$ (see Wang and
Zhu~\cite{WZ13}), we know that $\frac{D_{n}}{D_{n-1}}$ is increasing
and it is easy to obtain $\frac{D_{n}}{D_{n-1}}>15$ for $n\geq 30$.
Thus, taking $r(n)=15$ and $s(n)=\frac{16n^3 - 24n^2 + 12n -
2}{n^3}$, we can get
\begin{eqnarray*}&&8(a_{n+1}a_n+b_{n+1})r(n)+5
(a_{n+1}b_{n}-a_{n-1}a_{n+1}a_n-a_{n-1}b_{n+1})\\
&=&5\left[ 24(a_{n+1}a_n+b_{n+1})+
a_{n+1}b_{n}-a_{n-1}a_{n+1}a_n-a_{n-1}b_{n+1}\right]>0
\end{eqnarray*}
and $15a_n+b_n>0$ for $n\geq 1$ by Maple.

Further, substituting $a_n=\frac{2(2n+1)(5n(n+1)+2)}{(n+1)^3}$ and
$b_{n}=\frac{-64n^3}{(n+1)^3}$ into (\ref{Function}) and using
Maple, we can easily get that $f''(15)>0$ and $f'(15)>0$ for
$n\geq181$, and $f(s(n))<0$ for $n\geq1$. By Theorem \ref{thm-}, we
have $\{\frac{D_{n+1}}{D_n}\}_{n\ge 181}$ is ratio log-convex. For
$0\leq n\leq 181$, it can be checked that
\begin{eqnarray*} D_{n+2}D_{n-2}D_n^6-D_{n+1}^4D_{n-1}^4>0
\end{eqnarray*}
by Maple. Thus, $\{\frac{D_{n+1}}{D_n}\}_{n\ge 0}$ is ratio
log-convex. This completes the proof.
\end{proof}
\end{ex}
\section{Infinite Log-monotonicity}
In this section, we proceed to infinite log-monotonicity of some
sequences.  The derangements number $d_n$ is log-convex and ratio
log-concave, see \cite{LW07} and \cite{CGW} respectively. Thus, by
Proposition~\ref{prop+d}, we know that $\{d_n\}_{n\geq 3}$ is
log-monotonic of order $3$. In fact, we can demonstrate that
$\{d_n\}_{n\ge 3}$ is asymptotically infinitely log-monotonic. In
order to do so, we need following two results. One is the known
result
\begin{eqnarray}\label{e}
|d_{n}-\frac{n!}{e}|\leq\frac{1}{2}
\end{eqnarray} for $n\geq 3$ and the other is that $\{\Gamma(n)\}_{n\geq1}$ is  strictly infinitely log-monotonic, see Chen et al. \cite{CGW1}.

\begin{thm}
The sequence of the derangements numbers $\{d_n\}_{n\ge 3}$ is an
asymptotically infinitely logarithmically monotonic sequence.
\end{thm}
\begin{proof}
From (\ref{e}), we can deduce
$$\frac{n!}{e}-\frac{1}{2}\leq d_n\leq \frac{n!}{e}+\frac{1}{2},$$
which implies $$\Gamma(n+1)-\frac{3}{2}\leq e d_n\leq
\Gamma(n+1)+\frac{3}{2}.$$ Thus, we have
\begin{eqnarray*}
&&e^2(d_{n+1}d_{n-1}-d_n^2)\\
&\geq&[\Gamma(n+2)-1.5][\Gamma(n)-1.5]-[\Gamma(n+1)+1.5]^2\\
&=&(\Gamma(n))^2\left\{\left[(n+1)n-\frac{1.5}{(n-1)!}\right]\left[1-\frac{1.5}{(n-1)!}\right]-\left[n+\frac{1.5}{(n-1)!}\right]^2\right\}\\
&=&(\Gamma(n))^2\left\{n-\frac{1.5(n^2-n+1)}{(n-1)!}\right\}\\
&>0&
\end{eqnarray*}
for $n\ge4$, which implies that $\{d_n\}_{n\ge 4}$ is log-convex.
\begin{eqnarray*}
&&e^4(d_{n+1}^3d_{n-1}-d_{n}^3d_{n+2})\\
&\geq&[\Gamma(n+2)-1.5]^3[\Gamma(n)-1.5]-[\Gamma(n+1)+1.5]^3[\Gamma(n+3)+1.5]\\
&=&\Gamma(n)(\Gamma(n+1))^3\left\{\left[n+1-\frac{1.5}{n!}\right]^3\left[1-\frac{1.5}{(n-1)!}\right]-\left[1+\frac{1.5}{n!}\right]^3\left[(n+2)(n+1)n+\frac{1.5}{(n-1)!}\right]\right\}\\
&>0&
\end{eqnarray*}
for $n\ge8$, implies that $\{d_n\}_{n\ge 8}$ is ratio log-concave.
Note that $\{\Gamma(n)\}_{n\geq1}$ is strictly infinitely
log-monotonic. So, similarly, it can be proceeded to the higher
order log-monotonicity. Thus, for any positive integer $k $, by the
sign-preserving property of limits, we can obtain that there exists
a positive $N$ such that the sequence $R^r\{d_n\}_{n\geq N}$ is
log-concave for $r$ odd and not greater than $k-1$ and is log-convex
for $r$ even and not greater than $k-1$. Thus, the sequence of the
derangements numbers $\{d_n\}_{n\ge 2}$ is asymptotically infinitely
log-monotonic.
\end{proof}
Many sequences of binomial coefficients share various log-behavior
properties, see Tanny and Zuker \cite{TZ74,TZ76}, Su and
Wang~\cite{SW08} for instance. In particular, Su and Wang proved
that $\binom{dn}{\delta n}$ is log-convex in $n$ for positive
integers $d>\delta$. Recently, Chen {\it et al.} \cite{CGW1} proved
that both the Catalan numbers $\frac{1}{n+1}\binom{2n}{n}$ and the
central binomial coefficients $\binom{2n}{n}$ are infinitely
log-monotonic.
 Now we can give a generalization as follows.
\begin{thm}\label{thm+e}
Let $n_{0},k_{0},\overline{k_0}$ be nonnegative integers and
$a,b,\overline{b}$ be positive integers. Define the sequence
\begin{equation*}\label{ai}
    C_i=\frac{(n_{0}+ia)!}{(k_0+ib)!(\overline{k_0}+i\overline{b})!},\qquad i=0,1,2,\ldots.
\end{equation*}
If $a\geq b+\overline{b}$ and $-1\leq k_0-(n_0+1)b/a\leq0$, then
$\{C_i\}_{i\geq0}$ is infinitely log-monotonic.
\end{thm}
\begin{proof}
In order to prove this theorem, we will need a result of Chen {\it
et al.} \cite{CGW1} as follows: Assume that a function $f(x)$ such
that $[\log f(x)]''$ is completely monotonic for $x \geq 1$ and $a_n
= f(n)$ for $n \geq 1$. Then the sequence $\{a_n\}_{n\geq1}$ is
infinitely log-monotonic. In what follows, we will apply this result
to our proof. Since
\begin{equation*}
 C_i=\frac{(n_{0}+ia)!}{(k_0+ib)!(\overline{k_0}+i\overline{b})!}
 =\frac{\Gamma(n_{0}+ai+1)}{{\Gamma(k_{0}+bi+1)}
{\Gamma(\overline{k_{0}}+i\overline{b}+1)}},
\end{equation*}
we define a function
\begin{eqnarray*}
g(x)=\log\frac{\Gamma(n_{0}+ax+1)}{{\Gamma(k_{0}+bx+1)}
{\Gamma(\overline{k_{0}}+x\overline{b}+1)}}.
\end{eqnarray*}
Thus, we can obtain that
\begin{eqnarray}\label{eq}
&&[g(x)]^{(n)}\nonumber\\
&=&[\log\Gamma(n_{0}+ax+1)]^{(n)}-[\log{\Gamma(k_{0}+bx+1)}]^{(n)}-[\log{\Gamma(\overline{k_{0}}+x\overline{b}+1)}]^{(n)}\nonumber\\
&=&(-1)^n\int_{0}^{\infty}\frac{t^{n-1}}{1-e^{-t}}\left[a^ne^{-t(n_{0}+ax+1)}-b^ne^{-t(k_{0}+bx+1)}-\overline{b}^ne^{-t(\overline{k_{0}}+x\overline{b}+1)}\right]dt\\
&=&(-1)^n\int_{0}^{\infty}a^nt^{n-1}e^{-tax}\left[\frac{e^{-(n_0+1)t}}{1-e^{-t}}-\frac{e^{-ta(k_{0}+1)/b}}{1-e^{-at/b}}-\frac{e^{-ta(\overline{k_{0}}+1)/\overline{b}}}{1-e^{-at/\overline{b}}}\right]dt\nonumber
\end{eqnarray}
since
$$[\log\Gamma(x)]^{(n)}=(-1)^{n}\int_{0}^{\infty}\frac{t^{n-1}e^{-tx}}{1-e^{-t}}dt$$ for
$x>0$ and $n\geq2$, see~\cite[p.16]{MOS66} for instance.

It follows from $a>b>0$ that for further simplification denote
$u=k_0-(n_0+1)b/a$, $p=a/b$, and $q=a/\overline{b}$. Clearly,
$\frac{1}{p}+\frac{1}{q}\leq1$. So we deduce that
\begin{eqnarray}\label{eqq}
(-1)^n[g(x)]^{(n)}
&=&\int_{0}^{\infty}a^nt^{n-1}e^{-t(n_{0}+ax+1)}h(t,u)dt,
\end{eqnarray}
where
$$h(t,u)=\frac{1}{1-e^{-t}}-\frac{e^{-tp(u+1)}}{1-e^{-pt}}-\frac{e^{uqt}}{1-e^{-qt}}.$$

Furthermore, we have the next claim for $-1\leq
k_0-(n_0+1)b/a\leq0$.
\begin{cl}
If $-1\leq u\leq 0$, then $h(t,u)>0$.
\end{cl}
\textbf{Proof of Claim:} It is obvious that $h(t, u)$ is concave in
$u$. Thus it suffices to show $h(t, u)> 0$ for $u =-1$ and $u = 0$.
Setting $u = 0$ since the case $u =-1$ can be obtained by switching
the roles of $p$ and $q$, we have
\begin{eqnarray*}
h(t,0)&=&\frac{e^{-t}}{1-e^{-t}}-\frac{e^{-tp}}{1-e^{-pt}}-\frac{e^{-qt}}{1-e^{-qt}}.
\end{eqnarray*}
Noting for $s>0$ that function $$f(s)=\frac{se^{-s}}{1-e^{-s}}$$
strictly decreases in $s$ and $\frac{1}{p}+\frac{1}{q}\leq1$, we
have
\begin{eqnarray*}
h(t,0)&\geq&(\frac{1}{p}+\frac{1}{q})\frac{e^{-t}}{1-e^{-t}}-\frac{e^{-tp}}{1-e^{-pt}}-\frac{e^{-qt}}{1-e^{-qt}}\\
&=&\frac{f(t)-f(tp)}{tp}+\frac{f(t)-f(tq)}{tq}\\
&\ge&0.
\end{eqnarray*}
 This completes the proof of this
Claim.

Thus, by (\ref{eqq}) and this Claim, we have $(-1)^n[g(x)]^{(n)}>0$,
which implies the infinite log-monotonicity of $C_i$ by the result
of Chen {\it et al.} \cite{CGW1}. This completes the proof.
\end{proof}
It follows from Theorem~\ref{thm+e} that the following two
corollaries are immediate.
\begin{cor}
Let $n_{0},k_{0},\D,\d$ be four nonnegative integers. Define the
sequence
\begin{equation*}\label{ai}
    C_i=\binom{n_{0}+i\D}{k_0+i\d},\qquad i=0,1,2,\ldots.
\end{equation*}
If $\D>\d>0$ and $-1\leq k_0-(n_0+1)\d/\D\leq0$, then the sequence
$\{C_n\}_{n\geq0}$ infinitely log-monotonic.
\end{cor}
For integer $p\geq 2$, Fuss-Catalan numbers~\cite{HP} are given by
the formula
$$C_p(n)=\frac{1}{(p-1)n+1}\binom{pn}{n}=\frac{1}{pn+1}\binom{pn+1}{n}=\frac{(pn)!}{((p-1)n+1)!n!}.$$
It is well known that the Fuss-Catalan numbers count the number of
paths in the integer lattice $\mathbb{Z}\times \mathbb{Z}$ (with
directed vertices from $(i,j)$ to either $(i,j+1)$ or $(i+1,j))$
from the origin $(0,0)$ to $(n,(p-1)n)$ which never go above the
diagonal $(p-1)x=y$.
\begin{cor}
The Fuss-Catalan sequence $\{C_p(n)\}_{n\geq0}$ is  infinitely
log-monotonic.
\end{cor}


\end{document}